\documentclass[11pt,a4paper,oneside,reqno]{amsart}

\usepackage{amsmath, amsthm, amsopn, amssymb}
\usepackage{mathtools, thmtools, thm-restate}
\usepackage[ruled,vlined]{algorithm2e}
\usepackage{bbm}
\usepackage{enumerate}
\usepackage{enumitem}
\usepackage{hyperref}
\usepackage{stmaryrd} 
\usepackage{xcolor}

\usepackage{tikz}
\usetikzlibrary{topaths,calc}
\usepackage{subcaption}

\usepackage{geometry}
\declaretheorem[name=Theorem,within=section]{thm}
\newtheorem{lemma}[thm]{Lemma}
\newtheorem{prop}[thm]{Proposition}

\theoremstyle{definition}
\newtheorem*{remark}{Remark}

\renewcommand{\Pr}{\mathbb{P}}

\newcommand{\eps}{\varepsilon}

\newcommand{\br}[1]{\llbracket{#1}\rrbracket}

\renewcommand{\le}{\leqslant}
\renewcommand{\ge}{\geqslant}

\title{On the anti-Ramsey threshold}

\author{Eden Kuperwasser}
\thanks{This research was supported by the ERC Consolidator Grant 101044123 (RandomHypGra)}
\address{School of Mathematical Sciences, Tel Aviv University, Tel Aviv 6997801, Israel}
\email{kuperwasser@tauex.tau.ac.il}

\begin{document}

\maketitle

\begin{abstract}
We say that a graph $G$ is anti-Ramsey for a graph $H$ if any proper edge-colouring of $G$ yields a rainbow copy of $H$, i.e.\ a copy of $H$ whose edges all receive different colours. In this work we determine the threshold at which the binomial random graph becomes anti-Ramsey for any fixed graph $H$, given that $H$ is sufficiently dense. Our proof employs a graph decomposition lemma in the style of the Nine Dragon Tree theorem that may be of independent interest.
\end{abstract}

\section{Introduction}

Let $G$ and $H$ be graphs. We say that $G$ is \emph{Ramsey} for $H$ if any $2$-colouring of the edges of $G$ admits a monochromatic copy of $H$. We say that $G$ is \emph{anti-Ramsey} for $H$ if any proper colouring of the edges of $G$ admits a \emph{rainbow} copy of $H$, i.e.\ a copy whose edges all have distinct colours. 

The first investigations into the Ramsey and anti-Ramsey behaviours of random graphs were both motivated by a search for locally sparse graphs that satisfy these properties. Frankl and R\"odl~\cite{FraRod86} proved an upper bound on the threshold that $G_{n,p}$ is Ramsey for triangles in order to find an efficient construction of $K_4$-free graphs that were still Ramsey for $K_3$. A few years later R\"odl and Tuza~\cite{RodTuz92} found graphs with arbitrarily large girth that are anti-Ramsey for cycles by upper bounding the threshold where $G_{n,p}$ becomes anti-Ramsey for them. However, while R\"odl and Ruci\'nski~\cite{RodRuc95} proved the Random Ramsey theorem in 1995, locating the Ramsey threshold for any graph $H$, nearly three decades have passed and the anti-Ramsey threshold remains unresolved.

One basic intuition for the threshold --- common to both these properties --- is for $G_{n,p}$ to have about as many copies of $H$ as it does edges. Indeed, if we view the copies as constraints and the edges as corresponding to degrees of freedom, then we might expect that $G_{n,p}$ has the anti-Ramsey property when the number of copies is much larger than the number of edges. The threshold indicated by this intuition is readily captured by the \emph{maximum $2$-density}
 \[m_2(H) \coloneqq \max \left\{ \frac{e_{H'} - 1}{v_{H'} - 2} \colon H' \subseteq H, v_{H'} \ge 3 \right\} \cup \left\{ \frac{1}{2} \right\} ,\] 
and indeed the content of the Random Ramsey theorem is that the threshold that $G_{n,p}$ is Ramsey for $H$ (for almost\footnote{The exception being star forests. If one excludes the path on four vertices as well this threshold is shown to be semi-sharp.} any $H$) is at $p \sim n^{-1/m_2(H)}$. The following result of Kohayakawa, Konstadinidis, and Mota~\cite{KohKonMot14} gives a similar upper bound for the anti-Ramsey threshold.

\begin{thm}[\cite{KohKonMot14}]
\label{thm:1-statement}
For any graph $H$ there is a constant $c_1$ such that \[ \lim_{n \to \infty} \Pr \left( \text{$G_{n,p}$ is anti-Ramsey for $H$} \right) = 1\] whenever $p \ge c_1 \cdot n^{-1/m_2(H)}$. 
\end{thm}

However, this upper bound is not always tight. For example, the anti-Ramsey threshold for $K_3$ coincides with the threshold that $G_{n,p}$ contains even one copy of the triangle, since any proper edge colouring of the triangle must be rainbow. Of course, this happens much earlier than when the number of triangles rivals the number of edges.

We can generalize this example in the following way. For every graph $G$ that is anti-Ramsey for $H$, the threshold for $G_{n,p}$ containing a copy of $G$ is another upper bound for the anti-Ramsey threshold. Recall that the threshold for containing a graph $G$ is at $p \sim n^{-1/m(G)}$ where we define the \emph{maximal density} \[ m(G) \coloneqq \max \left\{ \frac{e_{G'}}{v_{G'}} \colon \emptyset \neq G' \subseteq G \right\},\] and that for any $c_0 > 0$, we have $\lim_{n\to \infty} \Pr(G_{n,p} \supseteq G) > 0$ when $p = c_0 \cdot n^{-1/m(G)}$. Therefore, to have any hope of establishing a semi-sharp threshold  for the anti-Ramsey property at $n^{-1/m_2(H)}$, it is necessary that every $G$ that is anti-Ramsey for $H$ has $m(G) > m_2(H)$.

We say that $H$ is strictly $2$-balanced if $m_2(H) > m_2(H')$ for any proper $H' \subset H$. The next result shows that this necessary condition is actually sufficient for strictly $2$-balanced graphs.

\begin{thm}[Nenadov--Person--\v{S}kori\'c--Steger~\cite{NenPerSkoSte17}, Behague--Hancock--Hyde--Letzter--Morrison~\cite{BehHanHydLetMor24}]
\label{thm:reduction}
Let $H$ be a strictly $2$-balanced graph. If any $G$ which is anti-Ramsey for $H$ must have $m(G) > m_2(H)$, then there exists a constant $c_0$ such that \[ \lim_{n \to \infty} \Pr \left( \text{$G_{n,p}$ is anti-Ramsey for $H$} \right) = 0\] whenever $p \le c_0 \cdot n^{-1/m_2(H)}$. 
\end{thm}

In other words, this theorem reduces the lower bound to the verification of the completely deterministic necessary condition. This condition was verified in~\cite{NenPerSkoSte17} for cliques on at least $19$ vertices and cycles on at least $7$ vertices. Later works extended the threshold statement for these graphs, obtaining a slightly weaker lower bound\footnote{These results prove a weaker version of the sufficient condition, showing that any $G$ which is anti-Ramsey for the graph $H$ has $m(G) \ge m_2(H)$. The proof of Theorem~\ref{thm:reduction} then implies the result for $p \ll n^{-1/m_2(H)}$.} for cliques~\cite{KohMotParSch23} and cycles~\cite{BarCavMotPar19}, both on $5$ or more vertices. In~\cite{BehHanHydLetMor24} the condition was verified for strictly $2$-balanced graphs $H$ with $1 < m_2(H) < \frac{\delta(H)(\delta(H) + 1)}{2\delta(H) + 1}$, a family that includes all $D$-regular graphs with at least $4D$ vertices.

On the other hand, the triangle is far from being the only counterexample. The previously mentioned works on cliques~\cite{KohMotParSch23} and cycles~\cite{BarCavMotPar19} showed that $K_4, C_4,$ and $K_4$ minus an edge are also thus afflicted. Actually, there are infinitely many counterexamples since forests are known to be problematic~\cite{ColKohMorMot22}. But even beyond that, it was shown in~\cite{AraMarMatMenMorMot22,KohKonMot18} that one can construct infinitely many other such graphs $H$ by taking any $H'$ with $1 < m_2(H') < 2$ and gluing an arbitrary number of triangles on one of its edges. Interestingly, all known counterexamples have bounded maximum $2$-density. One can therefore ask whether the threshold is at $n^{-1/m_2(H)}$ for sufficiently dense graphs $H$.
\\

Our main result answers this question with a happy yes.

\begin{thm}[anti-Ramsey threshold]
\label{thm:threshold}
Any graph $H$ with $m_2(H) \ge 19$ has positive constants $c_0 < c_1$ for which \[ \lim_{n \rightarrow \infty} \Pr \left(\text{$G_{n,p}$ is anti-Ramsey for $H$}\right) = \left\{ \begin{array}{ll}
1, & p \ge c_1 \cdot n^{-1/m_2(H)}, \\
0, & p \le c_0 \cdot n^{-1/m_2(H)}.
\end{array} \right. \]
\end{thm}

To locate the threshold we verify the deterministic condition of Theorem~\ref{thm:reduction} with the next theorem. Our result actually yields a stronger colouring scheme; given a graph $G$ with $m(G) \ge 18$ we will find one colouring that simultaneously avoids rainbow copies of all $H$ with $m_2(H) \ge m(G)$.

\begin{thm}
\label{thm:deterministic}
Any graph $G$ with $m(G) \ge 18$ admits a proper edge-colouring where every rainbow subgraph has $2$-density strictly below $m(G)$. 
\end{thm}

\begin{remark}
We do not make a special effort to optimize the constants, since we cannot bridge the gap to the densest counterexamples. See the discussion after the proof of Theorem~\ref{thm:deterministic} for more details.
\end{remark}

Our proof hinges on a graph decomposition statement, Proposition~\ref{prop:degen-decomp}, which bares some similarities to the Nine Dragon Tree theorem~\cite{JiaYan17}, and may be of independent interest. It allows us to decompose the edges of $G$ into two low-degeneracy graphs one of which has bounded degree.

\subsection*{Acknowledgements}
The author would like to thank Wojciech Samotij for his careful reading and valuable comments on earlier drafts of this paper.

\section{The proof}

We begin by painting the general outline of our argument. Given graphs $G$ and $H$, suppose we want to properly colour the edges of $G$ without a rainbow copy of $H$. We aim to find a subgraph $B \subseteq G$ such that: \begin{enumerate}

\item the graph $B$ is $b$-\emph{bounded}, i.e.\ $\Delta(B) \le b$, and

\item any copy of $H$ in $G$ must use at least $b+2$ edges from $B$.
\end{enumerate}

If we manage to do that, we colour the edges of $G$ in the following way. First, use Vizing's theorem to properly colour the edges of $B$ with at most $b+1$ colours. Then, we give every edge of $G' = G \setminus B$ a new distinct colour. This colouring is proper, plus, as any copy of $H$ uses $b+2$ edges of $B$ it must therefore use some colour twice.

In the next section we present a decomposition lemma that will give us a sequence of bounded degree graphs $B$ whose complements are degenerate. Then, we will show how to utilize this degeneracy to obtain the second property mentioned above, which will allow us to prove Theorem~\ref{thm:deterministic}. As a result we will also prove Theorem~\ref{thm:threshold}.

\subsection{The decomposition} The following is the main technical result of the paper.

\begin{prop}[Degenerate decomposition]
\label{prop:degen-decomp}
Let $G$ be a graph and write $k = \lfloor m(G) \rfloor$ and $K = \lfloor 2m(G) \rfloor$. There are forests $F_K, \dots, F_{k+1}$ such that \begin{enumerate}[label=(\alph*)]
\item \label{item:bounded} For each $K \ge i \ge k+1$ the graph $F_i$ is $\lceil \frac{i}{i - m(G)} \rceil$-bounded.
\item \label{item:degen} For each $k \le j \le K$ the graph $B_j \coloneqq \bigcup_{i > j} F_i$ is $(K-j)$-degenerate, and the graph $G \setminus B_j$ is $j$-degenerate.
\end{enumerate}
\end{prop}

\begin{remark}
The degeneracy of $B_j$ will not play a part in our proof. We will only use the fact that it is bounded. Indeed, as a corollary of \ref{item:bounded} we have that $\Delta(B_j) \le \sum_{i>j} \Delta(F_i)$.
\end{remark}

Our decomposition relies on a repeated application of the next lemma.

\begin{lemma}
\label{lem:directed-decomp}
Let $J$ be a digraph and let $D = \Delta^+(J)$, and suppose that $D > m(J)$. Then there is an edge decomposition $J = J' \cup F$ where \[\text{$\Delta^+(J') \le D-1$, $\Delta^+(F) \le 1$, and $\Delta(F) \le \left\lceil \frac{D}{D - m(J)} \right\rceil$}.\]
\end{lemma}

The case $D = \lfloor m(J) \rfloor + 1$ appears in the work of Gao and Yang~\cite{GaoYan23} on digraph analogues of the Nine Dragon Tree theorem. To better understand this lemma we make a short comparison with the Nine Dragon theorem for psuedoforests, which we paraphrase here.

\begin{thm}[Fan--Li--Song--Yang~\cite{FanLiSonYan15}]
Let $G$ be a graph with $m(G) = k + \eps$, where $k$ is an integer and $\eps \in [0,1)$. Then there is an edge decomposition $G = G' \cup F$ and an orientation of the edges such that \[\text{$\Delta^+(G') \le k$, $\Delta^+(F) \le 1$, and $\Delta(F) \le \left\lceil \eps \cdot \frac{k+1}{1-\eps} \right\rceil$}.\]
\end{thm} 

Let $m(G) = k+\eps$ where $k = \lfloor m(G) \rfloor$. Using Hakimi's theorem~\cite{Hak65}, one can orient the edges of $G$ into a digraph $J$ with $\Delta^+(J) \le \lceil m(G) \rceil = k+1$. Then, the above lemma outputs a decomposition where the maximum total degree of $F$ is bounded by $\left\lceil \frac{k+1}{1-\eps} \right\rceil$. This falls short by a factor of $\eps$ when put against the bound in the theorem. 

There is a subtle difference however. In the lemma we are given a specific orientation and we stick with it in the decomposition, while in the Nine Dragons theorem one can choose the orientation of the base graph. This point will be crucial in our application as we will fix an orientation corresponding to a degenerate ordering, and respecting this orientation will ensure that the parts are also degenerate.  We proceed by proving Proposition~\ref{prop:degen-decomp} given the lemma, and supply its proof right afterwards.

\begin{proof}[Proof of Proposition~\ref{prop:degen-decomp}]

The average degree of every subgraph of $G$ is bounded from above by $2m(G)$, so every subgraph contains a vertex of degree at most $\lfloor 2m(G)\rfloor = K$. This in turn means that $G$ is $K$-degenerate. Let $J_K$ be the digraph obtained by directing the edges of $G$ according to the degenerate ordering, so that every vertex has out-degree at most $K$. 

By induction on $K \ge i \ge k+1$, and starting with $i = K$, we have $\Delta^+(J_i) \le i$. If $\Delta^+(J_i) \le i-1$ we set $J_{i-1} = J_i$ and set $F_i$ to be empty. Otherwise $\Delta^+(J_i) = i$ and we may apply Lemma~\ref{lem:directed-decomp} to $J_i$ as $m(J_i) \le m(G) < i$. Then, we get a decomposition $J_i = J_{i-1} \cup F_i$ where \[\text{$\Delta^+(J_{i-1}) \le i-1$, $\Delta^+(F_i) \le 1$, and $\Delta(F_i) \le \left\lceil \frac{i}{i - m(G)} \right\rceil$.}\] Note that the orientation of the edges is induced by a fixed ordering of the vertices. Therefore, the out-degree bound means that each $F_i$ is $1$-degenerate, or in other words a forest, completing the proof of~\ref{item:bounded}. Furthermore, the same consideration tells us that the graphs $B_{K-j} = F_K \cup \dots \cup F_{K-(j-1)}$ and $J_j = G\setminus B_j$ are $j$-degenerate for each $k \le j \le K$, thus proving~\ref{item:degen}.
\end{proof}

\begin{proof}[Proof of Lemma~\ref{lem:directed-decomp}]
Given a set $U \subseteq V(J)$ we write $N^+(U)$ for the out-neighbourhood of $H$, i.e.\ the set of vertices $v \in V(J)$ such that $uv \in J$ for some $u \in U$. Let $L$ denote the set of vertices whose out-degree is exactly $D$, and let $R = N^+(L)$. Given any positive integer $c$ we define an auxiliary bipartite graph $A$ where one side is $L$ and the other side is $R \times \br{c}$. For every pair of vertices $u \in L$ and $v \in R$ we add the edges between $u$ and each copy of $(v, i)$ for all $i \in \br{c}$.

We would like to apply Hall's theorem in order to find a matching from $L$ to $R \times \br{c}$ in $A$. Suppose to the contrary that Hall's condition does not hold. Then there is a subset $U \subseteq L$ of order $u$ whose neighbourhood in $A$ has fewer than $u$ vertices. Since that neighbourhood is of the form $N^+(U) \times \br{c}$, we find that $|N^{+}(U)| < \frac{u}{c}$. Now, consider the induced subgraph $I \coloneqq J[U \cup N^+(U)]$. It has at least $u \cdot D$ edges, and strictly less than $u + \frac{u}{c}$ vertices. As a result we have \[m(J) \ge \frac{e_{I}}{v_{I}} > \frac{D}{1 + 1/c},\] which implies \[1 + \frac{1}{c} > \frac{D}{m(J)} =  \frac{m(J) + (D - m(J))}{m(J)} = 1  + \frac{1}{m(J) / (D-m(J))},\] and therefore $c < \frac{m(J)}{D - m(J)}$. So setting $c = \left\lceil \frac{m(J)}{D - m(J)} \right\rceil$ we get that Hall's condition is met and that we have the desired matching.

Let $F$ be the set of edges $uv \in J$ such that our matching contains an edge $u(v,i) \in A$ for some $i \in \br{c}$. From the construction $\Delta^+(F) \le 1$ and $\Delta^-(F) \le c$. Indeed, any edge touching a vertex $u$ must originate from an edge in the matching that uses some vertex $(u,i)$ for $i\in \br{c}$, and since it is a matching any such vertex in $A$ is touched at most once. Thus, the maximum degree of $F$ is at most $c+1 = \left\lceil \frac{m(J)}{D - m(J)} + 1 \right\rceil = \left\lceil \frac{D}{D-m(J)} \right\rceil$ as required. Finally, since every vertex with out-degree $D$ has lost one out-edge to $F$, the remainder $J' = J \setminus F$ has $\Delta^+(J') \le D - 1$. 
\end{proof}

\subsection{Proofs of Theorems~\ref{thm:deterministic} and~\ref{thm:threshold}}

With Proposition~\ref{prop:degen-decomp} at hand we can find in $G$ some bounded degree graph $B$ where $G \setminus B$ is degenerate. To put our plan into action we will need to show that $G \setminus B$ is far from containing a copy of any $H$ of some prescribed $2$-density, and so it suffices for us to control how far $H$ is from being degenerate. The next short lemma handles this task.

\begin{lemma}
\label{lem:k-degen}
Let $H$ be a graph with $d_2(H) \coloneqq \frac{e_H - 1}{v_H - 2} \ge k + \eps$ where $k$ is an integer and $\eps \in [0,1)$, and let $H' \subseteq H$ be a $d$-degenerate subgraph. Then \[e_H - e_{H'} \ge \binom{d - 1}{2} - (d - k - \eps) \cdot (v_H - 2).\]
\end{lemma}

\begin{proof}
Let $r = e_H - e_{H'}$. Since $H'$ is $d$-degenerate we have $e_{H'} \le \binom{d}{2} +  d(v_H - d)$. Therefore we can write \[k + \eps \le \frac{e_H - 1}{v_H - 2} \le \frac{d (v_H - d) + \binom{d}{2} - 1 + r}{v_H - 2 }.\] Consequently $r \ge \binom{ d - 1}{2} - (d - k - \eps) \cdot (v_H - 2)$.
\end{proof}

We now make the following observations. Let $d_2(H) \ge k + \eps$. When plugging in $d = k$, the above lemma means that we need to remove at least $\binom{k-1}{2} + \eps \cdot (v_H -2)$ edges from $H$ to get a $k$-degenerate graph. If we plug in $d = k+1$ we need to remove at least $\binom{k}{2} - (1-\eps)\cdot(v_H - 2)$ many edges to make it $(k+1)$-degenerate. This goes to show that if $H$ is close to $(k+1)$-degenerate it must therefore have many vertices, which in turn means that it is even further away from being $k$-degenerate. Combined with the bounds we get from the decomposition lemma, we are now ready to prove the theorems. 

\begin{proof}[Proof of Theorem~\ref{thm:deterministic}]
We apply Proposition~\ref{prop:degen-decomp} to $G$ and get the forests $F_K, \dots, F_{k+1}$ where $k = \lfloor m(G) \rfloor$ and $K = \lfloor 2m(G) \rfloor$. We will properly colour the edges of $G$ according to the decomposition. First, for each $K \ge i \ge k+1$ the forest $F_i$ is bipartite and can be coloured using $\Delta(F_i) \le \lceil \frac{i}{i - m(G)} \rceil$ colours following K\"onig's theorem. We will assign a new set of colours for each forest, and then colour the remaining edges of $G$ by giving each edge a new distinct colour. 

Let $H$ be any graph with $d_2(H) \ge m(G)$. We wish to show that this colouring does not contain a rainbow copy of $H$. Set \[r \coloneqq \sum_{i = k+2}^{K} \left\lceil \frac{i}{i - m(G)} \right\rceil,\]  and suppose first that one needs to remove more than $r$ edges of $H$ to get a $(k+1)$-degenerate graph. Since the graph $G \setminus B_{k+1}$ is $(k+1)$-degenerate, it means that any copy of $H$ must use at least $r+1$ edges from $B_{k+1}$. However, by definition we used at most $r$ colours for the edges of $B_{k+1}$, so any copy of such a graph $H$ must use some colour twice and is therefore not rainbow.

Next, suppose that one can remove at most $r$ edges from $H$ and get a $(k+1)$-degenerate graph. Writing $\eps = m(G) - k$ we can apply Lemma~\ref{lem:k-degen} to get that
 \[r \ge \binom{k}{2} - (1-\eps)(v_H-2),\] and therefore \[v_H - 2 \ge \frac{\binom{k}{2} - r}{1 - \eps}.\] 
Now, the graph $G \setminus B_k$ is $k$-degenerate, so another application of Lemma~\ref{lem:k-degen} tells us that any copy of $H$ in $G$ must use at least \[\eps\cdot (v_H - 2) + \binom{k-1}{2} \ge \eps \cdot \frac{ \binom{k}{2} - r }{1-\eps}   + \binom{k-1}{2}\] edges from $B_k$. Again, if this number exceeds the number of colours used for $B_k$ then $H$ cannot be rainbow. 
To conclude the proof we claim that $\binom{k-1}{2} > k+ 2 + r$ for $k$ sufficiently large, since this would immediately imply the required \[
\begin{split}
\left\lceil \frac{k+1}{1 - \eps} \right \rceil + r & \le \frac{k+1}{1-\eps} +1 + r = (\frac{\eps}{1-\eps} + 1)(k+1) + 1 + r \\
&= \frac{\eps}{1-\eps} (k+1) + k + 2 + r < \frac{\eps}{1-\eps}\left( \binom{k-1}{2} - r - 1\right) + \binom{k-1}{2} \\
&\le \frac{\eps}{1-\eps} \left(\binom{k}{2} - r \right) + \binom{k-1}{2}.
\end{split}
\]

Our claim follows by using the fact that $K \le 2k + 1$ to bound \[
\begin{split}
r &=  \sum_{i = k+2}^{K} \left \lceil  1 + \frac{m(G)}{i - m(G)} \right \rceil \le \sum_{i=1}^{K-k-1} \left(2  + \frac{m(G)}{i + 1 - \eps} \right) \\ 
& \le  2(K - k - 1) + m(G) \cdot \sum_{i = 1}^{K - k - 1} \frac{1}{i} \le 2k + (k+1)\cdot(1 + \log k).
\end{split}\] Therefore, whenever $k \ge 18$ we have $\binom{k-1}{2} > k+2 + r$. 
\end{proof}

\begin{remark}
It is evident that a more careful analysis would yield a better constant, but we choose simplicity in the absence of a tight bound. We point out that the same framework could be more efficient for graphs $H$ with extra structure that are further from being degenerate, e.g.\ if they have higher minimum degree or because they have more vertices compared to $k$ and $\eps$.
\end{remark}

Using Theorem~\ref{thm:deterministic}, the proof of the main theorem is a short corollary.

\begin{proof}[Proof of Theorem~\ref{thm:threshold}]
Let $H$ be a graph with $m_2(H) \ge 19$. Let $H' \subseteq H$ be a strictly $2$-balanced with $m_2(H') = m_2(H)$. We will prove that the reduction of Theorem~\ref{thm:reduction} applies to $H'$, meaning that there is a constant $c_0 > 0$ such that $\text{$\Pr(G_{n,p}$ is anti-Ramsey for $H$)}$ tends to $0$ as $n$ tends to infinity, so long as $p \le c_0 \cdot n^{-1/m_2(H)}$. Of course, if there is a proper edge-colouring such that any copy of $H'$ is not rainbow, then certainly the same applies for $H \supseteq H'$.

Now, any graph $G$ with $m(G) \le m_2(H')$ has two options. Either $m(G) < 18$ in which case \[m_2(G) \le m(G) + 1 < 19 \le m_2(H')\] and therefore $G$ has no copies of $H'$ in it, or $m(G) \ge 18$ meaning we can apply Theorem~\ref{thm:deterministic} to get a proper edge-colouring of $G$ such that any rainbow subgraph has $2$-density strictly below $m(G) \le m_2(H')$. This in turn means that there is no rainbow copy of $H'$ as required. Theorem~\ref{thm:reduction} then allows us to find the constant $c_0$ for the $0$-statement of the threshold, as we discussed above. Since the $1$-statement is always true by Theorem~\ref{thm:1-statement} we have proved the theorem.
\end{proof}

\bibliographystyle{amsplain}
\bibliography{anti_ramsey} 

\providecommand{\bysame}{\leavevmode\hbox to3em{\hrulefill}\thinspace}
\providecommand{\MR}{\relax\ifhmode\unskip\space\fi MR }
\providecommand{\MRhref}[2]{%
  \href{http://www.ams.org/mathscinet-getitem?mr=#1}{#2}
}
\providecommand{\href}[2]{#2}
\begin{thebibliography}{10}

\bibitem{AraMarMatMenMorMot22}
Pedro Ara{\'u}jo, Ta{\'\i}sa Martins, Let{\'\i}cia Mattos, Walner
  Mendon{\c{c}}a, Luiz Moreira, and Guilherme~O Mota, \emph{On the anti-ramsey
  threshold for non-balanced graphs}, arXiv preprint arXiv:2201.05106 (2022).

\bibitem{BarCavMotPar19}
Gabriel~Ferreira Barros, Bruno~Pasqualotto Cavalar, Guilherme~Oliveira Mota,
  and Olaf Parczyk, \emph{Anti-ramsey threshold of cycles for sparse graphs},
  Electronic Notes in Theoretical Computer Science \textbf{346} (2019), 89--98.

\bibitem{BehHanHydLetMor24}
Natalie Behague, Robert Hancock, Joseph Hyde, Shoham Letzter, and Natasha
  Morrison, \emph{Thresholds for constrained ramsey and anti-ramsey problems},
  arXiv preprint arXiv:2401.06881 (2024).

\bibitem{ColKohMorMot22}
Maur{\'\i}cio Collares, Yoshiharu Kohayakawa, Carlos~Gustavo Moreira, and
  Guilherme~Oliveira Mota, \emph{The threshold for the constrained ramsey
  property}, arXiv preprint arXiv:2207.05201 (2022).

\bibitem{FanLiSonYan15}
Genghua Fan, Yan Li, Ning Song, and Daqing Yang, \emph{Decomposing a graph into
  pseudoforests with one having bounded degree}, Journal of Combinatorial
  Theory, Series B \textbf{115} (2015), 72--95.

\bibitem{FraRod86}
Peter Frankl and Vojtech R{\"o}dl, \emph{Large triangle-free subgraphs in
  graphs without {$K_4$}}, Graphs and Combinatorics \textbf{2} (1986), no.~1,
  135--144.

\bibitem{GaoYan23}
Hui Gao and Daqing Yang, \emph{Digraph analogues for the nine dragon tree
  conjecture}, Journal of Graph Theory \textbf{102} (2023), no.~3, 521--534.

\bibitem{Hak65}
S~Louis Hakimi, \emph{On the degrees of the vertices of a directed graph},
  Journal of the Franklin Institute \textbf{279} (1965), no.~4, 290--308.

\bibitem{JiaYan17}
Hongbi Jiang and Daqing Yang, \emph{Decomposing a graph into forests: the nine
  dragon tree conjecture is true}, Combinatorica \textbf{37} (2017), no.~6,
  1125--1137.

\bibitem{KohKonMot14}
Yoshiharu Kohayakawa, Pavlos~B Konstadinidis, and Guilherme~Oliveira Mota,
  \emph{On an anti-ramsey threshold for random graphs}, European Journal of
  Combinatorics \textbf{40} (2014), 26--41.

\bibitem{KohKonMot18}
\bysame, \emph{On an anti-ramsey threshold for sparse graphs with one
  triangle}, Journal of Graph Theory \textbf{87} (2018), no.~2, 176--187.

\bibitem{KohMotParSch23}
Yoshiharu Kohayakawa, Guilherme~Oliveira Mota, Olaf Parczyk, and Jakob
  Schnitzer, \emph{The anti-ramsey threshold of complete graphs}, Discrete
  Mathematics \textbf{346} (2023), no.~5, 113343.

\bibitem{NenPerSkoSte17}
Rajko Nenadov, Yury Person, Nemanja {\v{S}}kori{\'c}, and Angelika Steger,
  \emph{An algorithmic framework for obtaining lower bounds for random ramsey
  problems}, Journal of Combinatorial Theory, Series B \textbf{124} (2017),
  1--38.

\bibitem{RodRuc95}
Vojt{\v{e}}ch R{\"o}dl and Andrzej Ruci{\'n}ski, \emph{Threshold functions for
  ramsey properties}, Journal of the American Mathematical Society \textbf{8}
  (1995), no.~4, 917--942.

\bibitem{RodTuz92}
Vojt{\v{e}}ch R{\"o}dl and Zsolt Tuza, \emph{Rainbow subgraphs in properly
  edge-colored graphs}, Random Structures \& Algorithms \textbf{3} (1992),
  no.~2, 175--182.

\end{thebibliography}
\appendix

\end{document}